\documentclass[leqno,12pt]{article} 

\usepackage[T2A]{fontenc}
\usepackage[cp1251]{inputenc}
\usepackage[english]{babel}
\usepackage{amsfonts}
\usepackage{amssymb,amsmath,amstext}
\usepackage{amscd, amsthm, latexsym}
\usepackage{color}
\usepackage[final]{graphicx}
\usepackage{epsfig}
\usepackage{euscript}
\usepackage{mathrsfs}
\usepackage{ifthen}
\usepackage{multirow}
\usepackage{array}
\usepackage{rotating}

\def\hide#1{}
\def\old#1{}

\def\oop#1{}
\def\gap#1{}

\theoremstyle{plain}
\newtheorem{theorem}{Theorem}
\newtheorem{proposition}{Proposition}
\newtheorem*{theorem*}{Theorem}
\newtheorem{lemma}{Lemma}
\newtheorem{remark}{Remark}
\newtheorem{corollary}{Corollary}
\newtheorem{definition}{Definition}

\makeatletter

\def\address#1#2{\begingroup
\noindent\parbox[t]{7.8cm}{%
\small{\scshape\ignorespaces#1}\par\vskip1ex
\noindent\small{\itshape E-mail address}%
\/: #2\par\vskip4ex}\hfill%
\endgroup}%

\makeatother

\title{A characterization of monotypically supersymmetric
polynomials}
\author{
Grigory~Chelnokov
{\small\em \; National Research University Higher School of Economics, Moscow, Russia} \\
 Maxim~Turevskii {\small\em \; SPBSU,facultee of mathematics and computer science}}
\date{}

\begin{document}

\newcounter{figcounter}
\setcounter{figcounter}{0} \addtocounter{figcounter}{1}

\maketitle 
\begin{abstract}
We introduce an object that has obvious similarity to the classical
one - the algebra of supersymmetric polynomials. Despite the
similarity, the known structure theorems on supersymmetric
polynomials do not help in the study of the new object, so we prove
their counterpart for the new object.

\end{abstract}


\section*{Introduction}
The study of supersymmetric polynomials (as well as almost all
objects with the prefix super-) is motivated by needs of theoretical
physics. A polynomial $p$ in $F[x_1,\dots,x_m,y_1,\dots,y_n]$ is
called {\em supersymmetric} if the following three conditions hold:
\begin{itemize}
\item[(1)] $p$ is invariant under permutations of $x_1,\dots,x_m$;
\item[(2)] $p$ is invariant under permutations of $y_1,\dots,y_n$;
\item[(3)] when the substitution $x_1=t,\;y_1=-t$ is made in $p$,
the  resulting polynomial is independent of $t$.
\end{itemize}
Paper \cite{bi-supsym2} provides explicit linear basis of algebra
$T_{(m,n)}$ of supersymmetric polynomials, \cite{bi-supsym} provides
some relations of supersymmetric and super Schur polynomials thus
connecting supersymmetric polynomials with representations of Lie
superalgebra $sl(m/n)$.

In this context, it is natural to wonder, what object one gets by
replacing in the condition (3) the set of pairs $K_{m,n}$ with
another highly symmetric graph. The first coming  to mind candidate
is, certainly, the complete graph $K_n$. So, the

\begin{definition}\label{T(n)} A polynomial $p$ in $F[x_1,\dots,x_n]$ is called
{\em monotypically supersymmetric} if the following two conditions
hold:
\begin{itemize}
\item[(1)] $p$ is invariant under permutations of $x_1,\dots,x_n$;
\item[(2)] when the substitution $x_1=t,\;x_2=-t$ is made in $p$,
the  resulting polynomial is independent of $t$.
\end{itemize}
By $T_n$ denote the algebra of monotypically supersymmetric
polynomials in $n$ variables. By $T_n^d$ and $T_n^{d-}$ denote
subspaces in $T_n$ of degree $d$ and at most $d$ respectively.
\end{definition}

Although monotypical supersymmetry does not have immediately obvious
applications in theoretical physics, it is still an object worthy of
research.

 Consider power sum polynomials $p_k=x_1^k+\cdots+x_n^k$.
Obviously, any power sum with an odd index is monotypically
supersymmetric $p_{2k+1}\in T_n$ (for integer non-negative $k$). In
this paper we prove that polynomials $p_{2k+1}$ generate $T_n$ as an
algebra and provide explicit canonical form for elements of $T_n$.


\subsection*{Initial notations}

Here we present only the minimum set of notations necessary to
smoothly formulate the main result. The bulk of the notation and
language conventions are two sections below.

$F$ always mean some field with $Char(F)=0$.

For $T_n, T_n^d, T_n^{d-}$ see Definition \ref{T(n)} above.

Similarly, $S_n, S_n^d, S_n^{d-}$ denote subalgebra of symmetric
polynomials, symmetric polynomials of degree $d$ and symmetric
polynomials of degree at most $d$ respectively.

Denote $p_k=x_1^k+\cdots+x_n^k$. Also, one special polynomial will
be important further: denote $\delta=\prod_{k,\neq\ell \in
[1..n]}(x_k+x_{\ell})$.


For a positive integer $n$ by $i(n)$ and $j(n)$ denote the maximal
integer numbers, such that $2i(n)\leqslant n$ and $2j(n)+1\leqslant
n$. We omit an argument when it is obvious from the context which
$n$ is meant.

Call an odd power sum $p_{2m+1}$ {\em elder} if $m>j$. Call a
product of odd power sums {\em long} if it consists of more then $i$
elder power sums. In other words, the product
$$
p_{2m_1+1}\cdots p_{2m_{i+1}+1},
$$
where $m_1,\dots,m_{i+1}>j$, is long.

Call a product of odd power sums {\em proper} if it does not
contains a long subproduct. In other words, the general form of
proper product is
$$
p_{1}^{c_1}p_{3}^{c_3}\cdots p_{2j+1}^{c_{2j+1}}p_{2m_1+1}\cdots
p_{2m_{i'}+1}
$$
where $i'\leqslant i$ and $m_k>j$ for $k\in [1..i']$.



\section*{Main result}
The main goal of this paper is to prove the following theorem.
%

\begin{theorem}\label{main}
The algebra of monotypically supersymmetric polynomials $T_n$ is
generated by odd power sums $p_{2k+1}$ for integer non-negative $k$.
Moreover proper products form a linear basis of $T_n$.
\end{theorem}



The remainder of the paper is devoted to the proof of Theorem
\ref{main}.


\subsection*{Notations, language conventions and main constructions}


When a polynomial is denoted by one of letters $f$, $g$, $\alpha$ or
$\beta$, this notation is used locally, in the next paragraph any of
this letters may denote another object. Contrary, if some polynomial
is denoted by any letter other then $f, g, \alpha, \beta$ we
emphasize that this polynomial is also used in some other argument
in this article.


Few times in this paper we consider $F[x_1,\dots,x_n]$ as tower of
extensions $F[x_2,x_4,\dots,x_{2i}][x_1,x_3,\dots,x_{2j+1}]$. In
this case we speak about {\em a coefficient} at
$x_1^{c_1}x_3^{c_3}\dots x_{2j+1}^{c_{2j+1}}$  in $f\in
F[x_2,x_4,\dots,x_{2i}][x_1,x_3,\dots,x_{2j+1}]$ (for a given
$(j+1)$-plet of powers $(c_1,\dots,c_{2j+1})$), meaning a
corresponding polynomial in $F[x_2,\dots,x_{2i}]$. Most likely, this
is used in conjunction with the following construction.

Recall that symmetric polynomials have unique representation via
power sums. For $f\in S_n$ by $Repr(f)$ denote a polynomial in
$F[x_1,\dots,x_n]$, such that $f=Repr(f)(p_1,\dots,p_n)$. We always
consider $Repr(f)$ as an element of
$F[x_2,x_4,\dots,x_{2i}][x_1,x_3,\dots,x_{2j+1}]$. Obviously, $f\to
Repr(f)$ is an isomorphism $S_n\to F[x_1,\dots,x_n]$.

Let $f\in S_n$ and $c_1,\dots,c_{2j+1}$ be a symmetric polynomals
and some $(2j+1)$-plet of non-negative integers. Let $g$ be the
coefficient at $x_1^{c_1}x_3^{c_3}\dots x_{2j+1}^{c_{2j+1}}$ in
$Repr(f)$. Call the polynomial
$f_{c_1,\dots,c_{2j+1}}=g(p_2,\dots,p_{2i})$ {\em a multiplier} of
$p_1^{c_1}p_3^{c_3}\dots p_{2j+1}^{c_{2j+1}}$ in $f$, or multiplier
at $(c_1,\dots,c_{2j+1})$ in $f$ for short. Informally speaking, a
multiplier is the corresponding coefficient after back substitution
of power sums.

%
%

For positive integers $n,k$ with $k\geqslant j(n)$ by $q_{(n,k)}$
and $r_{(n,k)}$ denote the multiplier of $(0,\dots,0)$ and
$(0,\dots,0,1)$ in $p_k$ respectively.

\subsection*{Preliminary remarks}

\begin{remark} It is sufficient to prove Theorem \ref{main} for
homogenous polynomials. Further, everywhere we are proving some
polynomial $f$ is representable via power sums we consider $f$ to be
homogeneous, without specifically mentioning this.
\end{remark}

\begin{remark}\label{divizibility} If $f=f_1f_2\neq0$ and $f,f_1\in S_n$,
then also $f_2\in S_n$. Further this is used without mentioning.
\end{remark}

\begin{lemma}\label{lin_independence_of_homogenous_products}
Let $h_{(k,\ell})$ be the complete homogenous polynomial of degree
$\ell$ in variables $x_{1},...,x_{k}$. Then products of the form
$h_{(k,\ell_{1})}...h_{(k,\ell_{k})}$ $(\ell_{1} \geq ... \geq
\ell_{k})$ form a basis of symmetric polynomials in
$x_{1},...,x_{k}$.
\end{lemma}
\begin{proof}
    In this proof we use one more notation: let $s_\lambda$ denote Schur polynomial,
    associated with an in integer partition $\lambda$ of length at most $k$. Use the first Jacobi-Trudi formula:
    $$
    s_{\lambda} = \det(h_{(k,\lambda_{i} + j - i)})_{i,j = 1}^{l(\lambda)} =
    \det\left[ \begin{matrix} h_{(k,\lambda_1)} & h_{(k,\lambda_1 + 1)} & \dots & h_{(k,\lambda_1 + k - 1)} \\
    h_{(k,\lambda_2-1)} & h_{(k,\lambda_2)} & \dots & h_{(k,\lambda_2+k-2)} \\
    \vdots & \vdots & \ddots & \vdots \\
    h_{(k,\lambda_k-k+1)} & h_{(k,\lambda_k-k+2)} & \dots & h_{(k,\lambda_k)} \end{matrix} \right].
    $$

 Thus,  $s_{\lambda}$ are linearly representable by the products of the form $h_{(k,\ell_{1})}...h_{(k,\ell_{k})}$.
 Schur polynomials form a linear basis in the space of symmetric polynomials. So, products $h_{(k,\ell_{1})}...h_{(k,\ell_{k})}$
  form a linear generating set, but for every degree $d$ the number of orbits of $d$-degree monomials is the number of
integer partitions of $d$ with maximal length $k$, thus equals the
number of
 degree $d$ products $h_{(k,\ell_{1})}...h_{(k,\ell_{k})}$. 
\end{proof}

\begin{lemma}\label{lin_independence}
Given a positive integer $n$,
consider products
$$
r_{(n,2m_{1}+1)}r_{(n,2m_{2}+1)}\cdots r_{(n,2m_{i}+1)}
$$
for all collections $(m_1,\dots,m_{j})$ where $j\leqslant
m_1\leqslant m_2 \leqslant\cdots \leqslant m_{i}$. Such products are
linearly independent.
\end{lemma}

\begin{remark} Since $r_{(n,j(n))}=1$ it is not important whether we
speak about the products of exactly $i$ $r$-s, or at most $i$. But
it is important here, that $r_{(n,k)}$ is not defined for $k<j(n)$.
\end{remark}

\begin{proof}[First proof.]
Just in this proof we need one more notation. By $\sigma_{(n,k)}\in
F[x_1,\dots,x_n]$ denote the elementary symmetric polynomial of
degree $k$ in variables $x_1,\dots,x_n$. For smaller number of
variables $n'<n$ we assume the natural embedding
$F[x_1,\dots,x_{n'}]\to F[x_1,\dots,x_n]$. Also, for $f\in S_n$ by
$\widetilde{f}$ denote the multiplier of $(0,\dots,0)$ in $f$.

For the sake of brevity denote $y_{\ell+i}=x_{\ell+i}+x_\ell$ for
$\ell \in [1..i]$, also in case $n$ is odd put $y_n=x_n$. Note that
for any $f\in F[x_1,\dots,x_n]$ holds the following:
$f(x_1,\dots,x_n)-f(x_1,\dots,x_i,x_{i+1}-y_{i+1},\dots,x_n-y_n)$
lies in the ideal $(y_{i+1},\dots,y_n)$. Denote this ideal by $I$
and use notation $f\simeq g$ as a shortcut for $f-g\in I$. We claim
\begin{itemize}
\item $p_{2k+1}\simeq 0$ and $\sigma_{(n,2k+1)} \simeq 0$ for any $k$;
\item $p_{2k}\simeq 2(x_1^{2k}+\cdots+x_i^{2k})$ and $\sigma_{(n,2k)} \simeq (-1)^k
\sigma_{(i,k)}(x_1^2,\dots,x_i^2)$ for any $k$;
\item for $f\in S_n$ holds $f\simeq \widetilde{f}$; indeed, $f- \widetilde{f} \in
(p_1,\dots,p_{2j+1})$, use the first point;
\item if two sequences of polynomials $\{f_k\}_{k=0}^\infty,\{g_k\}_{k=0}^\infty \in
F[x_1,\dots,x_n]$ are defined recurrently by
$$
f_k=\sum_{\ell \in [1,m]} \alpha_\ell f_{k-\ell} \quad \text{and}
\quad g_k=\sum_{\ell \in [1,m]} \beta_\ell g_{k-\ell} \quad \forall,
k\geqslant m
$$
then $\alpha_\ell\simeq \beta_\ell$ and $f_k\simeq g_k$ for $k<m$
guarantee $f_k\simeq g_k$ for all $k$.
\item if $f\neq 0 \in F(x_1,\dots x_i)$ then $f \not\simeq 0$.
\end{itemize}

We aim to calculate the sequence $r_{(n,2k+1)}$ up to equivalence $
\simeq$. Recall the recurrence for power sum series:
$$
\sum_{\ell \in [0..n]} (-1)^\ell \sigma_{(n,\ell)} p_{k-\ell}=0.
$$
This yields the following recurrence for $r_{(n,2k+1)}$:
$$
\sum_{\ell \in [0..n]} (-1)^\ell \widetilde{\sigma_{(n,\ell)}}
r_{(n,d-\ell)}+\sum_{\ell \in [0..n]} (-1)^\ell \alpha_\ell
q_{(n,d-\ell)}=0
$$
where $\alpha_\ell$ is the multiplier at $p_1^0\dots
p_{2j-1}^0p_{2j+1}^1$ in $\sigma_{(n,\ell)}$. Since we are
interested in $r_{(n,d-\ell)}$ up to equivalence, replace
$\widetilde{\sigma_{(n,\ell)}}$ with  $0$ for odd $\ell$ and
$(-1)^{\ell/2}\sigma_{(i,\ell/2)}(x_1^2,\dots,x_i^2)$ for even
$\ell$. Similarly, $q_{(n,k-\ell)}\simeq
2\big(x_1^{k-\ell}+\cdots+x_i^{k-\ell}\big)$ for even $k-\ell$, $0$
for odd. Next, $\alpha_\ell=0$ for $\ell<2j+1$ due to degree. To
prove $\alpha_{2j+1}=\frac{1}{2j+1}$ substitute different roots of
$x^{2j+1}-1=0$ to $x_1,\dots,x_{2j+1}$ (and $x_n=0$ in case $n$ is
even). In case of even $n$ it is unnecessary to calculate
$\alpha_{2j+2}$, since it is paired with zero $q_{(n,d-2j-2)}$ (we
are interested in odd $d$-s only). Replacing for convenience
$d=2k+2j+1$ come to
$$
\sum_{2\ell \in [0..n]} (-1)^\ell
\sigma_{(i,\ell)}(x_1^2,\dots,x_i^2)
r_{(n,2k+2j+1-2\ell)}-\frac{2}{2j+1}\big(x_1^{2k}+\cdots+x_i^{2k}\big)=0
$$
for $k>0$, $r_{n,2j'+1}=0$ for $j'<j$ and $r_{n,2j+1}=1$.

The solution of this recurrence is $r_{(n,2k+2j+1)} \simeq
\frac{2k+2j+1}{2j+1}h_{i,k}(x_{1}^{2},...,x_{i}^{2})$, this is
easily checked via induction on $k$. The application of Lemma
\ref{lin_independence_of_homogenous_products} and the last dot of
the claim in the beginning of this proof finishes it.

\end{proof}

%
%
%

\begin{corollary}
Proper products are linearly independent.
\end{corollary}

\begin{remark}\label{pi} The homomorphism $\pi: F[x_1,\dots,x_n] \to
F[x_1,\dots,x_{n-2}]$ given by $\pi(x_n)=\pi(x_{n-1})=0$ and
$\pi(x_k)=x_k$ for $k\in [1..n-2]$ maps $T_{n} \to T_{n-2}$, also it
maps power sums to power sum $\pi(p_k)=\overline{p_k}$ (we use upper
bar to discern polynomials in $F[x_1,\dots,x_{n-2}]$ from
polynomials in $F[x_1,\dots,x_n]$). This allows us to formulate the
following definition.
\end{remark}

\begin{definition}\label{preim}
Let $f\in T_{n}$ and assume $\pi(f)$ ($\pi$ as in Remark above) can
be represented through odd power sums, that is
$\pi(f)=g(\overline{p_1},\overline{p_3},\dots)$ for some $g\in
F[x_1,\dots]$. Then we call the polynomial $g(p_1,p_3,\dots)$ {\em a
canonical preimage of representation}
$\pi(f)=g(\overline{p_1},\overline{p_3},\dots)$ and denote it
$PreIm(f)$.
\end{definition}

Certainly, the equality $f=PreIm(f)$ does not necessarily holds, but
the following lemma provides some connection between $f$ and
$PreIm(f)$.

\begin{lemma}\label{div_delta}
Let  $f\in T_{n}$ and assume $\pi(f)$ can be represented through
power sums. Then $f-PreIm(f)$ is divisible by $\delta$.
\end{lemma}

\begin{proof} Polynomial $f-PreIm(f)$ have non-zero coefficient only at monomials with at most one zero exponent.
Indeed, by definition of $PreIm$, $f-PreIm(f)$ vanishes under
substitution $x_{n-1}=x_n=0$. So, each monomial contains one of
$x_{n-1},x_n$ with positive exponent. Due to symmetry, the same is
true for every pair $x_k,x_m$.

So, after substitution $x_1=t,\,x_2=-t$ the image of $f-PreIm(f)$ is
divisible by $t$, but it shouldn't depend on $t$. Both statements
are true only if the image is zero. The latter implies that
$f-PreIm(f)$ is divisible by $x_1+x_2$, use symmetry again.
\end{proof}

\begin{remark}\label{obvious} In case $deg(f)<\binom{n}{2}$ Lemma
\ref{div_delta} implies that $f-PreIm(f)$ (if defined) equals zero.
Indeed, $deg(PreIm(f))=deg(f)$ so $deg(f-PreIm(f))<deg(\delta)$,
thus divisibility implies $f-PreIm(f)=0$.
\end{remark}

One more nearly obvious statement
\begin{lemma}\label{obvious2}
For any symmetric polynomial $s \in S_n$ holds $\delta s \in T_n$.
\end{lemma}

\begin{proof} Indeed, $\delta s \in S_n$ because $\delta, s \in
S_n$. Under substitution $x_1=t,\;x_2=-t$ polynomial $\delta s$
vanishes, thus does not depends on $t$.
\end{proof}

We are going to prove Theorem \ref{main} via induction on $n$ by
steps $n \to n+2$, inside one step via one more induction on degree
of a polynomial. Next lemma provides the base of the outer
induction.

\begin{lemma}
Theorem \ref{main} holds true for the number of variables $n=2,3$.
\end{lemma}
\begin{proof} Consider case $n=2$. Note that the second condition in
the definition of monotypical supersymmetry is equivalent to
divisibility by $x_1+x_2$. Further, in case degree of $f\in T_2$ is
even, $f$ is divisible by $(x_1+x_2)^2$, since $f$ is representable
via $x_1+x_2$ and $x_1^2+x_2^2$, and must have even degree. Prove
the statement via induction on $d=deg(f)$, the base $d=1$ is
obvious. Next, if $d$ is even, represent $f=(x_1+x_2)f_1, \;\;
f_1\in T_2$ and use induction hypothesis. Assume $d$ is odd, then
$x_1^d+x_2^d$ is divisible by $x_1+x_2$ but not divisible by
$(x_1+x_2)^2$. So, there exists $\alpha \in F$ such that
$f-\alpha(x_1^d+x_2^d)$ is divisible by $(x_1+x_2)^2$ (certainly,
even by $(x_1+x_2)^3$). Again, represent $f-\alpha(x_1^d+x_2^d) =
(x_1+x_2)f_1, \;\; f_1\in T_2$ and use  induction hypothesis.

Consider case $n=3$. We prove two statements. First,
$\delta(x_1^2+x_2^2+x_3^2)^k$ for any integer non-negative $k$ is
representable as linear combination of proper products. Secondly,
the previous statement implies that any $f\in T_3$ is representable.

First statement prove by induction on $k$. For $k=0$ it holds since
$\delta=\frac{p_1^3-p_3}{3}$. For arbitrary $k$ consider
$f=p_{2k+3}$, then $PreIm(f)=p_1^{2k+3}$, by Lemma \ref{div_delta}
divisibility holds $p_{2k+3}-p_1^{2k+3}=\delta g$, then $g \in S_3$.
Apply the representation of $g$ through $p_1,p_2,p_3$:
$$
p_{2k+3}-p_1^{2k+3}=\delta \sum_{(k_1,k_2,k_3):\,
k_1+2k_2+3k_3=2k}\alpha_{k_1,k_2,k_3} p_1^{k_1}p_2^{k_2}p_3^{k_3}.
$$
The induction hypothesis applies to every term on the right-hand
side except $\alpha_{0,k,0}p_2^{k}$. So, if $\alpha_{0,k,0}\neq 0$
we proved that $\delta p_2^{k}$ is representable through proper
products.

To show $\alpha_{0,k,0}\neq 0$ consider the substitution $x_1=\phi$,
$x_2=\psi$, $x_3=-\phi-\psi$. Under it $p_1$ vanishes, $p_3$ is
divisible by $\phi$, so does $\delta$. Further, left-hand side is
not divisible by $\phi^2$. The sole term in the right-hand side, not
divisible by $\phi^2$ is $\delta\alpha_{0,k,0}p_2^{k}$, thus
$\alpha_{0,k,0}\neq 0$.

To prove the second statement consider $f\in T_3$, calculate
$PreIm(f)$ (simply said, $\beta(x_1+x_2+x_3)^{deg(f)}$) where $\beta
\in F$ is the coefficient at $x_1^{deg(f)}$ in $f$. Just as above
$$
f-PreIm(f)=\delta \sum_{(k_1,k_2,k_3):\,
k_1+2k_2+3k_3=2k}\beta_{k_1,k_2,k_3} p_1^{k_1}p_2^{k_2}p_3^{k_3}.
$$
Apply the first statement separately to each term
$\beta_{k_1,k_2,k_3}p_2^{k_2}\delta$ and multiply achieved
representation by $p_1^{k_1}p_3^{k_3}$.
\end{proof}

The base for an inner induction is trivial, since any symmetric
polynomial of degree 1 is $\alpha(x_1+\cdots+x_n)$.

\begin{lemma}\label{positive_defree_trick}
Let $f=p_{2a_1+1}\cdots p_{2a_{i+1}+1}$ be a long product, and
assume $f$ is equal to some linear combination of some proper
products
$$
f=\alpha_1g_1+\cdots+\alpha_{m}g_m, \eqno{*}
$$ where $\alpha_1,\dots\alpha_m \in F$. Then all proper products $g_{\ell}=p_{1}^{c_1}p_{3}^{c_3}\cdots p_{2j+1}^{c_{2j+1}}p_{2b_1+1}\cdots
p_{2b_{i'}+1}$ with $\alpha_{\ell} \neq 0$ have at least one of the
numbers $c_1,c_3,\dots,c_{2j+1}$ positive.
\end{lemma}

Note that in fact all variables $b$ and $c$ should have double
subscripts, because they also depends upon $\ell$, but we've omitted
this to avoid index pandemonium.

\begin{proof} Assume the contrary. We are going to prove equality
(*) can not hold whence left and right part have different
multiplies at $p_1^{0}p_3^{0}\dots p_{2j-1}^0p_{2j+1}^{k}$ for some
$k\leqslant i$.
%

Consider left-hand side. All $p_{a_1+1},\dots, p_{2a_{i+1}+1}$ have
odd degree, so have zero multiplier at $(0,\dots,0)$, so
$p_{2a_1+1}\cdots p_{2a_{i+1}+1}$ have zero coefficients at
multiplier at $(0,\dots,0,k)$ for all $k<i+1$.

Consider right-hand side. Call a proper product $g_\ell$ {\em dull}
if at least one of exponents $c_1,\dots,c_{2j-1}$ is positive.
Calculating multiplier at coefficients at $(0,\dots,0,k)$ one may
ignore dull summands $g_\ell$ since their considered multiplier is
zero. Among all non-dull $g_\ell$ with non-zero $\alpha_\ell$ pick
one with minimal value of $c_{2j+1}+i'$. From now on fix
$k=c_{2j+1}+i'$ where $c_{2j+1},i'$ are from the picked $g_\ell$.
Note that $k\leqslant i$ since assumption claims the existence of
non-dull $g\ell$ with $c_{2j+1}=0$, and its $i'\leqslant i$. Call a
non-dull proper product {\em minimal} if $k=c_{2j+1}+i'$ for this
product. Similarly, all non-minimal $g_\ell$ have zero multiplier at
$(0,\dots,0,k)$. A minimal $g_\ell$ is uniquely determined by its
collection $b_1,\dots,b_{i'}$, and have the multiplier $\alpha_\ell
r(n,b_1)\cdots r(n,b_{i'})$ at $(0,\dots,0,k)$. By Lemma
\ref{lin_independence} such multiplier are linearly independent,
thus total sum equals zero implies all $\alpha_\ell=0$, which is the
contradiction with the definition of minimal $g_\ell$.
\end{proof}

\begin{proposition}\label{nonzero_in_delta}
If Theorem \ref{main} holds for $n-2$ variables, then for $n$
variables $\delta$ is representable via proper products. Moreover,
the representation is
$$
\delta=\alpha \big(p_{2j+1}^i-PreIm(p_{2j+1}^i)\big)
$$
for some $\alpha\in F$, consequently $\delta$ have zero multiplier
at $(0,\dots,0,k)$ if and only if $k=i$.
\end{proposition}
\begin{proof}
Consider $\pi (p_{2j+1}^i)=\overline{p_{2j+1}}^i$. If Theorem
\ref{main} holds for $n-2$ variables, there is a representation
though proper products (proper for $n-2$ variables)
$\overline{p_{2j+1}}^i=\alpha_1\overline{g_1}+\cdots+\alpha_{m}\overline{g_m}$.
Then by Lemma \ref{positive_defree_trick} each of $g$-s is divisible
by one of $\overline{p_{1}},\dots,\overline{p_{2j-1}}$. Use this
representation to build $PreIm(p_{2j+1}^i)$. So, in
$PreIm(p_{2j+1}^i)$ each summand is divisible by one of
$p_{1},\dots,p_{2j-1}$. Thus, $Repr(p_{2j+1}^i-PreIm(p_{2j+1}^i))$
have just one non-zero coefficient at monomials of the form
$x_1^{0}x_3^{0}\dots x_{2j-1}^0x_{2j+1}^{k}$ -- the one for $k=i$.
Also, by Lemma \ref{div_delta} $p_{2j+1}^i-PreIm(p_{2j+1}^i)$ is
divisible by $\delta$. Since they have equal degree,
$p_{2j+1}^i-PreIm(p_{2j+1}^i)=\beta\delta$ for $\beta\in F$,
$\beta=0$ is impossible since $p_{2j+1}^i-PreIm(p_{2j+1}^i)$ have
non-zero multiplier at $(0,\dots,0,i)$.
\end{proof}

Now we are done with preparations. 

\subsection*{Proof of Theorem \ref{main}}
We will need one more notation, used locally only in this
subsection.
Denote $L_{even}^d=S_n^d\cap \{f(p_2,p_4\dots,p_{2i})|\;f\in
F(x_1,x_2,\dots,x_{i})\}$ and $L_{odd}^d=S_n^d\cap \langle p_1S_n,
p_3S_n,\dots,p_{2j+1}S_n\rangle$ (here the $\langle \rangle$ denote
the linear hull). Obviously $S_n^d=L_{even}^d\oplus L_{odd}^d$, thus
$\delta S_n^d=\delta L_{even}^d\oplus \delta L_{odd}^d$.

\begin{remark}\label{obvious3} The polynomial $\delta f$ for $f\in L_{odd}^d$ have zero multiplier at $(0,\dots,0,i)$. Indeed, by Proposition
\ref{nonzero_in_delta}, $\delta$ have all zero coefficients at
$(0,\dots,0,i')$ for $i'<i$, and $f$ have zero multiplier at
$(0,\dots,0)$.
\end{remark}

Fix positive integers $n$ and $d$. Assume the statement of Theorem
\ref{main} holds true for the number of variables $n-2$, and also
holds for $T_n^{(d-1)-}$ (that is, every polynomial in
$T_n^{(d-1)-}$ can be represented as a linear combination of proper
products). We'll show that the statement holds for $T_n^d$.

We need to prove that arbitrary $f \in T_n$ is representable via
proper products. Note that $\pi(f)$ is representable via proper
products of $\overline{p_1},\overline{p_3},\dots$
in virtue of Theorem
\ref{main} for $n-2$ variables. So, one does not need to specially
check conditions ``if can be represented'' in Lemmas
 \ref{div_delta}, \ref{positive_defree_trick}, and
Proposition \ref{nonzero_in_delta}.  Also $PreIm(f)$ is
representable via proper products. Thus it is sufficient to prove
that all polynomials of the form $f-PreIm(f)$ are representable,
applying Lemma \ref{div_delta} one needs that all elements of
$\delta S_n^{d-\binom{n}{2}}$ are representable. Note that all
elements of $\delta L_{odd}^{d-\binom{n}{2}}$  are representable.
Indeed, for a general element
$\delta(p_1s_1+\dots+p_{2j+1}s_{j+1})\in\delta
L_{odd}^{d-\binom{n}{2}}$ apply assumption of the induction to each
term $p_{2\ell+1}\delta s_{\ell+1}$. Since degree of $p_{2\ell+1}$
is greater then zero, $\delta s_{\ell+1}\in T_n^{(d-1)-}$, thus
representable. So, it is sufficient to prove that all cosets of
$\delta S_n^{d-\binom{n}{2}}/\delta L_{odd}^{d-\binom{n}{2}} \cong
\delta L_{even}^{d-\binom{n}{2}}$ are representable.

For this, consider all degree $d$ products of exactly $i$ power sums
$g_\ell=p_{2m_1+1}\cdots p_{2m_i+1}$ where $m_1,\dots,m_i\geqslant
j$ (that is, power sums are the last non-elder or elder). They are
representable by definition, also by Lemma \ref{lin_independence}
their multipliers at  $(0,\dots,0,i)$ are linearly independent, thus
$q_\ell$ themselves are linearly independent. Applying Lemma
\ref{positive_defree_trick} for $n-2$ variables to $\pi(g_\ell)$ get
that $PreIm(g_\ell)$ have zero multiplier at $(0,\dots,0,i)$
(indeed, at all $(0,\dots,0,k)$ for integer $k$ multiplier equals
zero). Thus polynomials $g_\ell-PreIm(g_\ell)$ are linearly
independent. As above, $g_\ell-PreIm(g_\ell) \in \delta S_n^d$.
Next, due to Remark \ref{obvious3} for any  $f \in\delta L_{odd}^d$,
polynomial $f$ have zero multiplier at $(0,\dots,0,i)$, thus
polynomials $g_\ell-PreIm(g_\ell)$ are linearly independent modulo
$\delta L_{odd}^{d-\binom{n}{2}}$, that is represent linearly
independent elements of $\delta L_{even}^{d-\binom{n}{2}}$. To
finish the prove one needs to show, that the amount of different
$g_\ell$ equals to $dim(\delta L_{even}^{d-\binom{n}{2}})$.

The amount of different $g_\ell$ equals $|\{(m_1,\dots,m_i):
j\leqslant m_1\leqslant\cdots\leqslant m_i, \;
(2m_1+1)+\cdots+(2m_i+1)=d\}$. Put $u_1=m_1-j,\dots u_i=m_i-j$ and
rewrite as $|\{(u_1,\dots,u_i): 0\leqslant
u_1\leqslant\cdots\leqslant u_i, \;
u_1+\cdots+u_i=\frac{d-i(2j+1)}{2}\}|$. Note that
$i(2j+1)=\binom{n}{2}$.

The dimension of linear space $L_{even}^{d-\binom{n}{2}}$ is the
cardinality of the set $\{(v_1,\dots,v_i):
deg(p_2^{v_1}p_4^{v_2}\cdots p_{2i}^{v_i})=d-\binom{n}{2}\}$. Covert
to $|\{(v_1,\dots,v_i):
v_1+2v_2+\cdots+iv_i=\frac{d-\binom{n}{2}}{2}\}|$. So we need
\begin{align*}
|\{(u_1,\dots,u_i): 0\leqslant u_1\leqslant\cdots\leqslant u_i, \;
u_1+\cdots+u_i=\frac{d-\binom{n}{2}}{2}\}|\\=|\{(v_1,\dots,v_i):
v_1+2v_2+\cdots+iv_i=\frac{d-\binom{n}{2}}{2}\}|.
\end{align*}
Equality holds because both parts counts the amount of Young diagram
of area $\frac{d-\binom{n}{2}}{2}$ with at most $i$ colons.

\bigskip

\address{G.~Chelnokov \\
 National Research University Higher School of Economics, Moscow, Russia
 \\}
 {\small\tt grishabenruven@yandex.ru }


\address{M.~Turevskii \\
 Federal State Budgetary Educational Institution of Higher Professional Education "Saint-Petersburg State University", Saint-Petersburg, Russia\\}
{ turmax20052005@gmail.com }


\begin{thebibliography}{6}

\bibitem{bi-supsym} Piotr Pragacz, Anders Thorup,
On a Jacobi-Trudi identity for supersymmetric polynomials, Advances
in Mathematics, Volume 95, Issue 1, 1992, Pages 8-17, ISSN
0001-8708, https://doi.org/10.1016/0001-8708(92)90042-J.
(https://www.sciencedirect.com/science/article/pii/000187089290042J)

\bibitem{bi-supsym2} John R. Stembridge, A characterization of supersymmetric
polynomials, Journal of Algebra, Volume 95, Issue 2, 1985, Pages
439-444, ISSN 0021-8693,
https://doi.org/10.1016/0021-8693(85)90115-2.
(https://www.sciencedirect.com/science/article/pii/0021869385901152)


\end{thebibliography}
\end{document}